\documentclass[a4paper,11pt]{amsart}
\usepackage{amssymb}
\usepackage{amsmath}
\usepackage{verbatim}
\usepackage{url}
\usepackage[all]{xy}

\usepackage{color}

\numberwithin{equation}{section}

\theoremstyle{plain}

\newtheorem{theorem}[equation]{Theorem}
\newtheorem{proposition}[equation]{Proposition}

\newtheorem{lemma}[equation]{Lemma}
\newtheorem{corollary}[equation]{Corollary}

\newtheorem{question}[equation]{Question}

\theoremstyle{definition}

\newtheorem{remark}[equation]{Remark}
\newtheorem{example}[equation]{Example}
\newtheorem{definition}[equation]{Definition}
\newtheorem{notation}[equation]{Notation}
\newtheorem{subsec}[equation]{}

\newcommand{\C}{{\mathbb{C}}}
\newcommand{\R}{{\mathbb{R}}}

\newcommand{\sA}{{\mathcal A}}
\newcommand{\sN}{{\mathcal{N}}}

\newcommand{\vk}{\varkappa}

\newcommand{\isoto}{\overset{\sim}{\to}}

\newcommand{\labelto}[1]{\xrightarrow{\makebox[1.5em]{\scriptsize ${#1}$}}}

\newcommand{\Spec}{{\rm Spec\,}}

\newcommand{\SAut}{{\rm SAut}}
\newcommand{\Aut}{{\rm Aut}}

\newcommand{\Gal}{{\rm Gal}}

\newcommand{\siga}{{\sigma_\gamma}}
\newcommand{\muga}{{\mu_\gamma}}

\newcommand{\ctil}{{\tilde c}}
\newcommand{\upgam}{\hs^\gamma\!}
\newcommand{\qs}{{\rm qs}}

\newcommand{\diag}{{\rm diag}}

\newcommand{\Gbar}{{\overline{G}}}
\newcommand{\Ztil}{{\widetilde{Z}}}

\newcommand{\hs}{\kern 0.75pt}

\newcommand{\emm}{\bfseries}

\newcommand{\X}{{{\sf X}}}

\newcommand{\Inn}{\mathrm{Inn}}

\newcommand{\id}{{\rm id}}

\newcommand{\Gtil}{{\widetilde{G}}}

\newcommand{\SL}{{\rm SL}}
\newcommand{\SU}{{\rm SU}}

\newcommand{\dmd}{{\text{\tiny$\diamondsuit$}}}

\newcommand{\cbar}{{c}}

\newcommand{\Tbar}{{\overline T}}
\newcommand{\Hbar}{{\overline H}}

\renewcommand{\gamma}{s}

\newcommand{\tr}{{\rm tr}}

\newcommand{\deltatil}{{\tilde \delta}}

\title[Models of $G$-varieties]%
{Existence of equivariant models of $G$-varieties}

\author{Mikhail Borovoi}

\address{Raymond and Beverly Sackler School of Mathematical Sciences,
Tel Aviv University, 6997801 Tel Aviv, Israel}

\email{borovoi@post.tau.ac.il}

\thanks{This research was partially supported by the Israel Science Foundation (grant No. 870/16)}



\date{\today}

\begin{document}

\begin{abstract}
Let $k_0$ be a field of characteristic 0, and let $k$ be a fixed algebraic closure of $k_0$.
Let $G$ be an algebraic $k$-group, and let $Y$ be a $G$-variety over $k$.
Let $G_0$ be a $k_0$-model ($k_0$-form) of $G$.
We ask whether $Y$ admits a $G_0$-equivariant $k_0$-model $Y_0$.

We assume that $Y$ admits a $G_\dmd$-equivariant $k_0$-model $Y_\dmd$,
where $G_\dmd$ is an inner form of $G_0$.
We give a Galois-cohomological criterion for the existence of
a $G_0$-equivariant $k_0$-model $Y_0$ of $Y$.
We apply this criterion to certain spherical homogeneous varieties $Y=G/H$.
\end{abstract}

\maketitle



\section{Introduction}
\label{s:intro}

\begin{subsec}\label{ss:intro-models}
Let $k_0$ be a field of characteristic 0,  and let $k$ be a fixed algebraic closure of $k_0$.
Set $\Gamma=\Gal(k/k_0)$.

Let $G$ be a connected algebraic  group over $k$ (not necessarily linear).
Let $Y$ be a $G$-variety, that is, an  irreducible algebraic variety over $k$
together with a morphism
\[\theta \colon G\times_k Y\to Y\]
defining an action of $G$ on $Y$.
We say that $(Y,\theta)$ is a {\em $G$-$k$-variety} or just that $Y$ is a $G$-$k$-variety.

Let $G_0$ be a {\em $k_0$-model} ($k_0$-form) of $G$, that is, an algebraic group over $k_0$
together with an isomorphism of algebraic $k$-groups
\[\nu_G\colon G_0\times_{k_0} k\isoto G.\]
By a {\em  $G_0$-equivariant $k_0$-model of the $G$-$k$-variety $(Y,\theta)$}
we mean a $G_0$-$k_0$-variety $(Y_0,\theta_0)$
together with an isomorphism $\nu_Y\colon Y_0\times_{k_0} k\isoto Y$
such that the following diagram commutes:
\begin{equation*}
\xymatrix{
G_{0,k}\times_k Y_{0,k}\ar[r]^-{\theta_{0,k}}\ar[d]_{\nu_G\times\nu_Y}   & Y_{0,k}\ar[d]^{\nu_Y} \\
G\times_k Y\ar[r]^-\theta                                                &Y
}
\end{equation*}

Inspired by the works of Akhiezer and Cupit-Foutou  \cite{ACF}, \cite{Akhiezer},
for a given $k_0$-model $G_0$  of $G$ we ask
whether there exists a $G_0$-equivariant $k_0$-model $Y_0$ of $Y$.
\end{subsec}

\begin{subsec}\label{ss:intro-twist}
With the above notation,
we consider the group $\Aut(G)$ of automorphisms of $G$.
We regard $\Aut(G)$ as an abstract group.
Any $g\in G(k)$ defines an {\em inner automorphism}
\[i_g\colon G\to G,\quad x\mapsto gxg^{-1}\text{ for }x\in G(k).\]
We obtain a homomorphism
\[i\colon G(k)\to \Aut(G).\]
We denote by $\Inn(G)\subset  \Aut(G)$ the image of the homomorphism $i$ and
we say that $\Inn(G)$ is the group of inner automorphisms of $G$.
We may identify $\Inn(G)$ with $\Gbar(k)$, where
\[\Gbar=G/Z(G)\]
and $Z(G)$ is the center of $G$.

Let $G_\dmd$ be a $k_0$-model of $G$.
We write $Z_\dmd$ for the center $Z(G_\dmd)$,
then $\Gbar_\dmd:=G_\dmd/Z_\dmd$ is a $k_0$-model of $\Gbar$.
Let $\cbar\colon \Gamma\to \Gbar_\dmd(k)$ be a {\em 1-cocycle}, that is, a locally constant map
such that the following cocycle condition is satisfied:
\begin{equation}\label{e:cocycle-cond}
\cbar_{\gamma t}=\cbar_\gamma\cdot\upgam\cbar_t\quad \text{for all }\gamma,t\in\Gamma.
\end{equation}
We denote the set of such 1-cocycles by $Z^1(\Gamma,\Gbar_\dmd(k))$ or by $Z^1(k_0,\Gbar_\dmd)$.
For $c\in Z^1(k_0,\Gbar)$ one can define the \emph{$\cbar$-twisted inner form}
$_\cbar(G_\dmd)$ of $G_\dmd$; see Subsection \ref{ss:tw} below.
For simplicity we write $_\cbar G_\dmd$ for $_\cbar(G_\dmd)$.
\end{subsec}

\begin{subsec}\label{ss:intro-qs}
It is well known that if $G$ is a connected reductive $k$-group,
then any $k_0$-model $G_0$ of $G$ is an inner form of a quasi-split model;
see e.g., Springer \cite[Proposition 16.4.9]{Springer2}.
In other words, there exist a quasi-split model $G_\qs$ of $G$
and a 1-cocycle $\cbar\in Z^1(k_0, \Gbar_\qs)$
such that $G_0=\hs_\cbar G_\qs$.
In some cases it is clear that $Y$ admits a $G_\qs$-equivariant $k_0$-model.
For example, assume that $Y=G/U$, where $U=R_u(B)$, the unipotent radical of a Borel subgroup $B$ of $G$.
Since $G_\qs$ is a \emph{quasi-split} model, there exists
a Borel subgroup $B_\qs\subset G_\qs$ (defined over $k_0$).
Set $U_\qs=R_u(B_\qs)$, then $G_\qs/U_\qs$ is a $G_\qs$-equivariant $k_0$-model of $Y=G/U$.
\end{subsec}

\begin{subsec}\label{ss:intro-we-ask}
In the setting of \ref{ss:intro-models} and \ref{ss:intro-twist}, let $G_\dmd$ be a $k_0$-model of $G$
and let $G_0=\hs_\cbar G_\dmd$, where $\cbar\in Z^1(k_0, \Gbar_\dmd)$.
Motivated by \ref{ss:intro-qs},   \emph{we assume}
that $Y$ admits a $G_\dmd$-equivariant $k_0$-model $Y_\dmd$,
and we ask whether $Y$ admits a $G_0$-equivariant $k_0$-model $Y_0$\hs.

We consider the short exact sequence
\[1\to Z_\dmd\to G_\dmd\to \Gbar_\dmd\to 1\]
and the connecting map
\[\delta\colon H^1(k_0,\Gbar_\dmd)\to H^2(k_0,Z_\dmd);\]
see Serre \cite[I.5.7, Proposition 43]{Serre}.
If $\cbar\in Z^1(k_0,\Gbar_\dmd)$, we write $[\cbar]$ for the corresponding cohomology class in $H^1(k_0,\Gbar_\dmd)$.
By abuse of notation we write $\delta[\cbar]$ for $\delta([\cbar])$.

We consider the group $\sA:=\Aut^G(Y)$ of $G$-equivariant automorphisms of $Y$,  which
we regard as an abstract group.
The $G_\dmd$-equivariant $k_0$-model $Y_\dmd$ of $Y$ defines a $\Gamma$-action on $\sA$, see Subsection \ref{ss:horo-2} below,
and we denote the obtained $\Gamma$-group by $\sA_\dmd$.
One can define the second Galois cohomology set  $H^2(\Gamma,\sA_\dmd)$.
See Springer \cite[1.14]{Springer-H2} for a definition of  $H^2(\Gamma,\sA_\dmd)$
in the case when the $\Gamma$-group $\sA_\dmd$ is nonabelian.

For $z\in Z_\dmd(k)$ we consider the $G$-equivariant automorphism
\[y\mapsto z\cdot y\colon\ Y\to Y.\]
We obtain a $\Gamma$-equivariant homomorphism
\[\vk\colon Z_\dmd(k)\to\sA_\dmd\hs,\]
which induces a map
\[\vk_*\colon H^2(k_0,Z_\dmd)\to H^2(\Gamma,\sA_\dmd).\]
\end{subsec}

\begin{theorem}[Theorem \ref{t:twist}]
\label{t:twist'}
Let $k,\ G,\ Y,\ k_0,\ G_\dmd,\ Y_\dmd,\  \sA_\dmd,\ \delta,\ \vk_*$
be as in Subsections \ref{ss:intro-models} and \ref{ss:intro-we-ask}.
In particular, we assume  that $Y$ admits a $G_\dmd$-equivariant $k_0$-model $Y_\dmd$.
We assume also that $Y$ is quasi-projective.
Let $\cbar\in Z^1(k_0,\Gbar_\dmd)$ be a 1-cocycle, and consider its class $[\cbar]\in H^1(k_0,\Gbar_\dmd)$.
Set $G_0=\hs_c G_\dmd$ (the inner twisted form of $G_\dmd$ defined by the 1-cocycle $c$).
Then the $G$-variety $Y$ admits a $G_0$-equivariant $k_0$-model if and only if the cohomology class
\[\vk_*(\delta[c])\in H^2(\Gamma, \sA_\dmd)\]
is neutral.
\end{theorem}

\begin{remark}
In the case when $\sA_\dmd$ is abelian, the condition ``$\vk_*(\delta[\cbar])$ is neutral''
means that $\vk_*(\delta[\cbar])=1$.
\end{remark}

Theorem \ref{t:twist'} is the main result of this paper.
Theorems \ref{t:tits'}, \ref{t:H-times-H'}, and \ref{t:U'} below
are applications of Theorem \ref{t:twist'} to the case when $Y=G/H$ is a homogeneous space of $G$.
In this case $\sA=A(k)$, where $A=\sN_{G}(H)/H$ and $\sN_{G}(H)$
denotes the normalizer of $H$ in $G$; see e.g. \cite[Lemma 5.1]{BG}.

In the following theorem, $G$ is a connected reductive group.

\begin{theorem}[Theorem \ref{t:tits}]
\label{t:tits'}
Let $G$ be a reductive group over an algebraically closed field $k$ of characteristic 0.
Let $H\subset G$ be a $k$-subgroup.
Let $k_0\subset k$ be a subfield such that $k$ is an algebraic closure of $k$.
Let $G_0$ be a $k_0$-model of $G$. Write
$G_0=\hs_c G_\qs$, where $G_\qs$ is a quasi-split $k_0$-model of $G$ and $c\in Z^1(k_0,G_\qs/Z(G_\qs)\hs)$.
Assume that
\begin{equation}\label{e:cond}
\text{$G/H$ admits a $G_\qs$-equivariant $k_0$-model  $Y_\qs$.}\tag{$*$}
\end{equation}
The $k_0$-model $Y_\qs$ defines a $k_0$-model $A_\qs=\Aut^{G_\qs}(Y_\qs)$ of $A=\Aut^G(Y)$.
Then $G/H$ admits a $G_0$-equivariant $k_0$-model $Y_0$ if and only if
the image in $H^2(k_0,A_\qs)$ of the Tits class $t(\Gtil_0)\in H^2(k_0,Z(\Gtil_\qs))$ is neutral
(see Section \ref{s:tits} below for the definition of the Tits class).
\end{theorem}

\begin{remark}\label{r:H1}
In Theorem \ref{t:tits'}, if there exists a $G_0$-equivariant $k_0$-model $Y_0$ of $G/H$,
then the set of isomorphism classes of such models is in a canonical bijection with the set $H^1(k_0, \Aut^{G_0}(Y_0))$.
\end{remark}

\begin{subsec}
Let
\[\ctil\colon \Gamma\to G_\dmd(k)\]
be a 1-cocycle \emph{with values in $G_\dmd$}, that is, $\ctil\in Z^1(k_0,G_\dmd)$.
Consider  $i\circ \ctil\in Z^1(k_0,\Gbar_\dmd)$, then
by abuse of notation we write $_\ctil G_\dmd$ for  $_{i\circ\ctil} G_\dmd$.
We say that $_\ctil G_\dmd$ is a \emph{pure inner form} of $G_\dmd$.
For a pure inner form $G_0=\hs_\ctil G_\dmd$, the $G$-variety $Y$
clearly admits a  $G_0$-equivariant $k_0$-model:
we may take $Y_0=\hs_\ctil Y_\dmd$; see Lemma \ref{l:pure-inner} below.
It follows from the cohomology exact sequence \eqref{e:coh-e-s} below
that for a cocycle $c\in Z^1(k_0, \Gbar_\dmd)$,
the twisted form $_c G_\dmd$ is a pure inner form of $G_\dmd$ if and only if $\delta[c]=1$.
\end{subsec}

\begin{subsec}\label{ss:H-times-H/Delta'}
Let $H$ be a connected linear $k$-group, and set $G=H\times_k H$.
Let $Y=H$, where $G$ acts on $Y$ by
\[ (h_1,h_2)*y=h_1\hs y \hs h_2^{-1}.\]
Note that $Y=G/\Delta$, where $\Delta\subset H\times_k H$ is the diagonal,
that is, $\Delta$ is $H$ embedded in $G$ diagonally.
Let $H^{(1)}_0$ and $H^{(2)}_0$ be two $k_0$-models of $H$.
We set $G_0=H^{(1)}_0\times_{k_0} H^{(2)}_0$ and ask whether $Y$ admits a $G_0$-equivariant $k_0$-model.
\end{subsec}

\begin{theorem}[Theorem \ref{t:H-times-H}]
\label{t:H-times-H'}
With the notation and assumptions of \ref{ss:H-times-H/Delta'},
$Y=(H\times H)/\Delta$ admits an $H^{(1)}_0\times_{k_0} H^{(2)}_0$-equivariant $k_0$-model
if and only if $H^{(2)}_0$ is a pure inner form of $H^{(1)}_0$.
\end{theorem}

\begin{example}
Let $k=\C$, $k_0=\R$, then $\Gamma=\{1,\gamma\}$, where $\gamma$ is the complex conjugation.
Let $H=\SL(4,\C)$. Consider the diagonal matrices
\[ I_4=\diag(1,1,1,1)\quad\text{and}\quad I_{2,2}=\diag(1,1,-1,-1).\]
Consider the real models $\SU(2,2)$ and $\SU(4)$ of $G$:
\begin{align*}
H^{(1)}_0&=\SU(2,2),\text{ where }\SU(2,2)(\R)=\{g\in \SL(4,\C)\ |\ g\cdot I_{2,2}\cdot\upgam g^\tr=I_{2,2}\},\\
H^{(2)}_0&=\SU(4), \text{ where }\SU(4)(\R)=\{g\in \SL(4,\C)\ |\ g\cdot I_4\cdot\upgam g^\tr=I_4\},
\end{align*}
where $g^\tr$ denotes the transpose of $g$.
Consider the 1-cocycle
\[c\colon\Gamma\to \SU(2,2)(\R),\quad 1\mapsto I_4,\ \gamma\mapsto I_{2,2}\hs.\]
A calculation shows that $_c \SU(2,2)\simeq \SU(4)$.
Thus $\SU(4)$ is a pure inner form of $\SU(2,2)$.
By Theorem \ref{t:H-times-H'}, there exists an $\SU(2,2)\times_\R \SU(4)$-equivariant
real model $Y_0$ of $Y=(H\times H)/\Delta.$
We describe this model explicitly.
We may  take for $Y_0$ the \emph{transporter}
\[Y_0=\{g\in \SL(4,\C)\ |\ g\cdot I_{4}\cdot \upgam g^\tr=I_{2,2}\}.\]
Clearly $Y_0$ is defined over $\R$.
It is well known that $Y_0$ is nonempty but it has no $\R$-points.
The group $G_0:=H^{(1)}_0\times_\R H^{(2)}_0$ acts on $Y_0$ by
\[(h_1,h_2)*g=h_1\hs g h_2^{-1}.\]
It is clear that $Y_0$ is a principal homogeneous space of both $H^{(1)}_0$ and $H^{(2)}_0$.
Thus $Y_0$ is a $G_0$-equivariant $k_0$-model of $Y$.
Compare \cite[Example 10.11]{BG}.
\end{example}

In the following theorem, $G$ is a connected reductive group and $Y=G/U$.

\begin{theorem}[Theorem \ref{t:U}]
\label{t:U'}
Let $k$ and $k_0$ be as in \ref{ss:intro-models},
and let $G$ be a  connected reductive group over $k$.
Let $B\subset G$ be a Borel subgroup, and write $U$ for the unipotent radical of $B$.
Consider the homogeneous space $Y=G/U$.
Let $G_0$ be a $k_0$-model of $G$.
Then $Y$ admits a $G_0$-equivariant $k_0$-model if and only if $G_0$
is a pure inner form of a quasi-split model of $G$.
\end{theorem}

\begin{example}
Let $k=\C$, $k_0=\R$, $G=\SL(4,\C)$, $Y=G/U$, where $U$ is as in Theorem \ref{t:U'}.
Let $G_0=\SU(4)$.
Since $G_0$ is a pure inner form of the quasi-split group $\SU(2,2)$,
by Theorem \ref{t:U'} the variety $G/U$ admits an $\SU(4)$-equivariant $\R$-model $Y_0$.
This model has no $\R$-points (because the stabilizer of an $\R$-point
would be a unipotent subgroup of $G_0$ defined over $\R$).
\end{example}

The plan for the rest of the paper is as follows.
In Section \ref{s:pre} we recall basic definitions and results.
In Section \ref{s:existence} we prove Theorem \ref{t:twist'}.
In Section \ref{s:tits} we prove Theorem \ref{t:tits'}.
In Section \ref{s:H x H} we prove Theorem \ref{ss:H-times-H/Delta'}.
In Section \ref{s:G/U} we prove Theorem \ref{t:U'}.

\textsc{Acknowledgements.}
The author is   grateful to Boris Kunyavski\u\i, Giuliano Gagliardi, Stephan Snigerov,
and Ronan Terpereau for stimulating discussions and/or e-mail exchanges.

\section{Preliminaries}
\label{s:pre}

\begin{subsec} \label{ss:semi-linear}
Let $k_0$, $k$, and $\Gamma$ be as in Subsection \ref{ss:intro-models}.
By a $k_0$-model of a $k$-scheme $Y$ we mean a $k_0$-scheme $Y_0$
together with an isomorphism of $k$-schemes
\[\nu_Y \colon Y_0\times_{k_0}k \isoto Y.\]

We write $\Gamma=\Gal(k/k_0)$.
For $\gamma \in \Gamma$, denote by $\gamma^{*} \colon \Spec k \to \Spec k$
the morphism of schemes induced by $\gamma$.
Notice that $(\gamma  t)^{*}=t^{*} \circ \gamma^{*}$.

Let $(Y,p_Y \colon Y\to \Spec k)$ be a $k$-scheme.
A \emph{$k/k_0$\,-semilinear automorphism of $Y$} is a pair $(\gamma,\mu)$
where $\gamma\in \Gamma$ and $\mu\colon Y\to Y$
is an isomorphism of schemes such that the diagram below commutes:
\begin{equation*}
\xymatrix@R=25pt@C=40pt{
Y\ar[r]^\mu \ar[d]_{p_Y}          & Y\ar[d]^{p_Y} \\
\Spec k\ar[r]^-{(\gamma^*)^{-1}}   & \Spec k
}
\end{equation*}
In this case we say also that $\mu$ is an $\gamma$-semilinear automorphism of $Y$.
We shorten ``$\gamma$-semilinear automorphism'' to ``$\gamma$-semi-automorphism''.
Note that if $(\gamma,\mu)$ is a semi-automorphism of $Y$, then $\mu$ uniquely determines $\gamma$;
see \cite[Lemma 1.2]{BG}.

We denote $\SAut(Y)$ the group of all $\gamma$-semilinear automorphisms $\mu$ of $Y$,
where $\gamma$ runs over $\Gamma=\Gal(k/k_0)$.
By a \emph{semilinear action} of $\Gamma$ on $Y$ we mean a homomorphism of groups
\[\mu \colon \Gamma \rightarrow \SAut(Y), \quad \gamma \mapsto \mu_\gamma\]
such that for each $\gamma \in \Gamma$, $\mu_\gamma$ is $\gamma$-semilinear.

If we have a $k_0$-scheme $Y_0$, then the formula
\begin{equation}\label{e:s-action}
\gamma \mapsto \id_{Y_0} \times (\gamma^{*})^{-1}
\end{equation}
defines a semilinear action of $\Gamma$ on
\[Y:=Y_0\times_{k_0}k=Y_0\times_{\Spec k_0}\Spec k.\]
Thus a $k_0$-model of $Y$ induces a semilinear action of $\Gamma$ on $Y$.

Let $(G,p_G \colon G \to \Spec k)$ be a $k$-group-scheme.
A $k/k_0$\,-semi-linear automorphism of $G$ is a pair
$(\gamma,\tau)$ where $\gamma\in \Gamma$ and $\tau\colon G\to G$
is a morphism of schemes such that the following diagram commutes
\begin{equation*}
\xymatrix@R=25pt@C=40pt{
G\ar[r]^\tau \ar[d]_{p_G}          & G\ar[d]^{p_G} \\
\Spec k\ar[r]^-{(\gamma^*)^{-1}}   & \Spec k
}
\end{equation*}
and the $k$-morphism
\[\tau_\natural\colon\gamma_*G\to G\]
is an isomorphism of algebraic groups over $k$; see \cite[Definition 2.2]{BG}
for the notations $\tau_\natural$ and $\gamma_* G$.

We denote  by $\SAut_{k/k_0}(G)$, or just by $\SAut(G)$,
the group of all $\gamma$-semilinear automorphisms
$\tau$ of $G$, where $\gamma$ runs over $\Gamma=\Gal(k/k_0)$.
By a semilinear action of $\Gamma$ on $G$ we mean a homomorphism
\[\sigma \colon \Gamma \to \SAut(G), \quad \gamma \mapsto \siga\]
such that for all $\gamma \in \Gamma$, $\siga$ is $\gamma$-semilinear.
As above, a $k_0$-model $G_0$ of $G$ induces a semilinear action of $\Gamma$ on $G$.

Let $G$ be an algebraic group over $k$ and let $Y$ be a $G$-$k$-variety.
Let $G_0$ be a $k_0$-model of $G$. It gives rise to a semilinear action
$\sigma \colon \Gamma \rightarrow \SAut(G),\gamma \mapsto \siga$.
Let $Y_0$ be a $G_0$-equivariant $k_0$-model of $Y$.
It gives rise to a semilinear action $\mu \colon \Gamma \to \SAut(Y)$ such that
for all $\gamma$ in $\Gamma$ we have
\begin{equation*}
\mu_\gamma(g \cdot y)=\sigma_\gamma(g) \cdot \mu_\gamma(y) \quad\text{for all } y\in Y(k),\ g\in G(k).
\end{equation*}
We say then that $\muga$ is \emph{$\siga$-equivariant.}
\end{subsec}

\begin{subsec}\label{ss:tw}
Let $k_0,\ k,$ and $\Gamma$ be as in Subsection \ref{ss:intro-models}.
Let $G_0$ be a $k_0$-model of $G$; it defines a semilinear action
\[ \sigma\colon \Gamma\to \SAut(G).\]
This action induces an action of $\Gamma$ on the abstract group $\Aut(G)$.
Recall that a map
\[c\colon \Gamma\to \Aut(G)\]
is called a \emph{1-cocycle} if the map $c$ is locally constant
and satisfies the cocycle condition \eqref{e:cocycle-cond}.
The set of such 1-cocycles is denoted by $Z^1(\Gamma,\Aut(G)\hs)$ or $Z^1(k_0,\Aut(G)\hs)$.
For $c\in Z^1(k_0,\Aut(G)\hs)$, we consider the $c$-twisted semilinear action
\[\sigma'\colon \Gamma\to \SAut(G),\quad \gamma\mapsto c_\gamma\circ \sigma_\gamma.\]
Then, clearly, $\sigma'_\gamma$ is an $\gamma$-semi-automorphism of $G$ for any $\gamma\in\Gamma$.
It follows from the cocycle condition \eqref{e:cocycle-cond} that
\[\sigma'_{\gamma t}=\sigma'_\gamma\circ\sigma'_t\ \text{ for all }\gamma,t\in\Gamma.\]
Since $G$ is an algebraic group, the semilinear action $\sigma'$ comes from some $k_0$-model $G_0'$ of $G$;
see Serre \cite[Section V.4.20, Corollary 2(ii) of Proposition 12]{Serre0}
and Serre \cite[III.1.3, Proposition 5]{Serre}.
We write $G_0'=\hs_c G_0$ and say that $G_0'$ is
the \emph{twisted form of $G_0$ defined by the 1-cocycle} $c$.
 \end{subsec}

\begin{lemma}\label{l:pure-inner}
Let $G$ be a linear algebraic group over $k$, and let $Y$ be a \emph{quasi-projective} $G$-$k$-variety.
Let $G_\dmd$ be a $k_0$-model of $G$, and assume that $Y$ admits a $G_\dmd$-equivariant $k_0$-model $Y_\dmd$.
Let $\ctil\in Z^1(k_0,G_\dmd)$ be a 1-cocycle.
Consider the pure inner form $_\ctil G_\dmd$. Then $Y$ admits a  $_\ctil G_\dmd$-equivariant $k_0$-model.
\end{lemma}

\begin{proof}
Write $G_0=\hs _\ctil G_\dmd$. We take $Y_0= _\ctil Y_\dmd$,
then $Y_0$ is a $G_0$-equivariant $k_0$-model.

We give details. The $k_0$-models $G_\dmd$ and $Y_\dmd$ define  semilinear actions
\[ \sigma\colon \Gamma\to\SAut(G)\quad\text{and}\quad\mu\colon\Gamma\to\SAut(Y)\]
such that for any $s\in\Gamma$ the semi-automorphism $\mu_s$ is $\sigma_s$-equivariant, that is,
\[\mu(g\cdot y)=\sigma_\gamma(g)\cdot\mu_\gamma(y)\quad\text{for all }g\in G(k),\ y\in Y(k).\]
Let $\ctil\colon \Gamma\to G(k)$ be a 1-cocycle, that is, $\ctil\in Z^1(k_0,G_\dmd)$.
Consider the pure inner form $G_0=\hs_\ctil G_\dmd$, then
\[\sigma_\gamma^0(g)=\ctil_\gamma\cdot \sigma_\gamma(g)\cdot \ctil_\gamma^{-1}
             \text{ for }\gamma\in\Gamma,\ g\in G(k).\]
where $\sigma^0$ is the semilinear action defined by $G_0$.
Now we define the twisted form $_\ctil Y_\dmd$ as follows.
We set
\[\mu_\gamma^0(y)=\ctil_\gamma\cdot\mu_\gamma^0(y).\]
Since $\ctil$ is a 1-cocycle, we have
\[\mu^0_{\gamma t}=\mu^0_\gamma\circ\mu^0_t\quad\text{for all }\gamma,t\in\Gamma.\]
Since $Y$ is quasi-projective, by Borel and Serre \cite[Lemme 2.12]{Borel-Serre}
the semilinear action $\mu^0\colon\Gamma\to\SAut(Y)$ defines a $k_0$-model $Y_0$ of $Y$.
An easy calculation shows that
\[\mu^0(g\cdot y)=\sigma^0_\gamma(g)\cdot\mu^0_\gamma(y)\quad\text{for all }g\in G(k),\ y\in Y(k),\]
hence by Galois descent we obtain an action of $G_0$ on $Y_0$ (defined over $k_0$);
see Jahnel \cite[Theorem 2.2(b)]{Jahnel}.
Thus $Y$ admits a $G_0$-equivariant $k_0$-model $Y_0=\hs_\ctil Y_\dmd$.
\end{proof}

\section{Model for an inner twist of the group}
\label{s:existence}

\begin{subsec}\label{ss:horo-1}
Let $k$ be an algebraically closed field of characteristic 0.
Let $G$ be an algebraic group over $k$.
Let $Y$ be a $G$-$k$-variety.
Let $Z(G)$ denote the center of $G$.
We consider the algebraic group $\Gbar:=G/Z(G)$.
The group $\Gbar(k)$ naturally acts on $G$:
\[g\hs Z(G)\colon\ x\mapsto g \hs x \hs g^{-1}\quad\text{for }g\hs Z(G)\in \Gbar(k),\ x\in G(k).\]

Let $k_0$ be a subfield of $k$ such that $k/k_0$ is an algebraic extension.
We write $\Gamma=\Gal(k/k_0)$, which is a profinite group.

Let $G_\dmd$ be a $k_0$-model of $G$.
We write $\Gbar_\dmd=G_\dmd/Z(G_\dmd)$, where $Z(G_\dmd)$ is the center of $G_\dmd$.
The $k_0$-model $\Gbar_\dmd$ of $\Gbar$ defines  a semilinear action:
\[\sigma\colon\Gamma\to \SAut(\Gbar);\]
cf.  \eqref{e:s-action}.
We write $\Gbar_\dmd(k)$ for the group of $k$-points of the algebraic $k_0$-group $\Gbar_\dmd$,
then we have an action of $\Gamma$ on $\Gbar_\dmd(k)$:
\[(s,g)\mapsto \upgam g=\sigma_s(g) \quad\text{for } s\in\Gamma,\ g\in \Gbar_\dmd(k)=\Gbar(k).\]

Let $c\in Z^1(k_0,\Gbar_\dmd)$ be a 1-cocycle, that is, a locally constant map
\[c\colon \Gamma\to \Gbar_\dmd(k)\quad\text{such that}
\quad c_{\gamma t}=c_\gamma\cdot\upgam c_t\quad\text{for all }\gamma,t\in\Gamma.\]
We denote by $G_0=\hs_c G_\dmd$ the corresponding inner twisted form of $G_\dmd$,
see Subsection \ref{ss:tw}.
This means that $G_0(k)=G_\dmd(k)$, but the Galois action is twisted by $c$:
\[\sigma^0_\gamma=c_\gamma\circ\sigma_\gamma\quad\text{for }\gamma\in \Gamma,\]
where we embed $\Gbar_\dmd(k)$ into $\Aut(G)$.

In this section we {\em assume} that there exists a $G_\dmd$-equivariant $k_0$-model $Y_\dmd$ of $Y$.
We give a criterion for the existence of a $G_0$-equivariant $k_0$-model $Y_0$ of $Y$,
where $G_0=\hs_c G_\dmd$.
\end{subsec}

\begin{subsec}\label{ss:horo-2}
We write $[c]\in H^1(k_0,\Gbar_\dmd)$ for the cohomology class of $c$.
We consider the short exact sequence
\[ 1\to Z(G_\dmd)\to G_\dmd\to \Gbar_\dmd\to 1\]
and the corresponding connecting map
\[\delta\colon H^1(k_0,\Gbar_\dmd)\to H^2(k_0,Z(G_\dmd))\]
from the cohomology exact sequence
\begin{equation}\label{e:coh-e-s}
H^1(k_0,Z(G_\dmd))\to H^1(k_0,G_\dmd)\to H^1(k_0,\Gbar_\dmd)\labelto{\delta}H^2(k_0,Z(G_\dmd));
\end{equation}
see Serre \cite[I.5.7, Proposition 43]{Serre}.
We obtain $\delta[c]\in H^2(k_0,Z(G_\dmd)$.

The $G_\dmd$-equivariant $k_0$-model $Y_\dmd$ of $Y$  defines
an action of $\Gamma$ on $\sA:=\Aut^G(Y)$ by
\[(\upgam a)(\upgam y)=\upgam(a(y))\quad \text{for }\gamma\in \Gamma,\ a\in\sA,\ y\in Y(k).\]
We denote by $\sA_\dmd$ the corresponding $\Gamma$-group.
We obtain homomorphisms
\[\mu\colon \Gamma\to\SAut(Y), \quad \gamma\mapsto \mu_\gamma,\quad\text{where }\mu_\gamma(y)=\upgam y\ \
\text{for }\gamma\in\Gamma,\ y\in Y(k)=Y_\dmd(k),\]
and
\begin{equation*}
\tau\colon \Gamma\to\Aut(\sA),\quad \gamma\mapsto\tau_\gamma\hs, \quad\text{where }\tau_\gamma(a)=\upgam a\ \
\text{for }\gamma\in\Gamma,\ a\in \sA.
\end{equation*}

The center $Z_\dmd\subset G_\dmd$ acts on $Y_\dmd$,
and this action clearly commutes with the action of $G_\dmd$.
Thus we obtain a canonical $\Gamma$-equivariant homomorphism
\[\vk\colon Z_\dmd(k)\to \sA_\dmd.\]
\end{subsec}

\begin{subsec}
We need  the nonabelian cohomology set $H^2(\Gamma,\sA_\dmd)$; see Springer \cite[1.14]{Springer-H2}.
Recall that an (abelian) 2-cocycle $z\in Z^2(k_0, Z_\dmd)$ is a locally constant map
\[a\colon \Gamma\times \Gamma\to Z_\dmd(k),\quad (s,t)\mapsto z_{s,t}\]
such that
\[\upgam d_{t,u}\cdot d_{s,tu}=d_{s,t}\cdot d_{st,u}\quad\text{for all }s,t,u\in \Gamma.\]
Then $\vk_*([z])\in H^2(\Gamma,\sA_\dmd)$ is by definition the class of the 2-cocycle $(\tau,\vk\circ z)$.
This class is called \emph{neutral} if there exists a locally constant map $a\colon \Gamma \to \sA_\dmd$
such that
\[ a_s\cdot \upgam a_t\cdot\vk(z_{s,t})\cdot a_{st}^{-1}=1\quad\text{for all }s,t\in\Gamma.\]
\end{subsec}

\begin{theorem}\label{t:twist}
Let $G,\ H,\ Y,\ k_0,\ G_\dmd,\ Y_\dmd,\  A_\dmd,\ \delta$
be as in Subsections \ref{ss:horo-1} and \ref{ss:horo-2}.
In particular we assume that  $Y$ admits a $G_\dmd$-equivariant $k_0$-model $Y_\dmd$.
We assume also that $Y$ is quasi-projective.
Let $c\in Z^1(k_0,\Gbar_\dmd)$ be a 1-cocycle, and consider it class $[c]\in H^1(k_0,\Gbar_\dmd)$.
Set $G_0=\hs_c G_\dmd$ (the inner twisted form of $G_\dmd$ defined by the 1-cocycle $c$).
The $G$-variety $Y$ admits a $G_0$-equivariant $k_0$-model if and only if the cohomology class
\[\vk_*(\delta[c])\in H^2(\Gamma, \sA_\dmd)\]
is neutral.
\end{theorem}

\begin{proof}
The $k_0$-model $G_\dmd$ of $G$ defines a homomorphism
\[\sigma\colon\Gamma\to \SAut(G),\quad \gamma \mapsto \sigma_\gamma\hs, \]
where each $\sigma_\gamma$ is an $\gamma$-semi-automorphism of $G$.
The $G_\dmd$-equivariant $k_0$-model $Y_\dmd$ of $Y$
defines a homomorphism
\[\mu\colon \Gamma\to\SAut(Y),\quad \gamma\mapsto\mu_\gamma\]
such that each $\mu_\gamma$ is an $\gamma$-semi-automorphism of $Y$
and is $\sigma_\gamma$-equivariant, that is,
\begin{equation}\label{e:mu-g-y}
\mu_\gamma(g\cdot y)=\sigma_\gamma(g)\cdot \mu_\gamma(y)\quad\text{for all }g\in G(k),\ y\in Y(k).
\end{equation}
Since the map $\gamma\mapsto\mu_\gamma$ is a homomorphism, we have
\begin{equation}\label{e:action-hom}
\mu_{\gamma t}=\mu_\gamma\circ\mu_t\quad\text{for all }\gamma,t\in\Gamma.
\end{equation}

We lift the 1-cocycle
\[c\colon\Gamma\to \Gbar(k)\]
to a locally constant map
\[\ctil\colon \Gamma\to G(k),\]
which does not have to be a 1-cocycle.
Let $\sigma^0\colon \Gamma\to\SAut(G)$ denote the homomorphism corresponding to the twisted form $G_0=\hs_cG_\dmd$,
then by definition
\[\sigma^0_\gamma(g)=\ctil_\gamma\cdot\sigma_\gamma(g)\cdot\ctil_\gamma^{-1}.\]

For $g\in G(k)$, we write $l(g)$ for the automorphism $y\mapsto g\cdot y$ of $Y$.
We have
\begin{equation}\label{e:g-a}
l(g)\circ a=a\circ l(g)\quad\text{for all } g\in G(k),\ a\in\sA_\dmd\hs,
\end{equation}
because $a$ is a $G$-equivariant automorphism of $Y$.
By \eqref{e:mu-g-y} we have $\mu_s(g\cdot y)=\sigma_s(g)\cdot \mu_s(y)$, hence
\begin{equation}\label{e:mu-g}
\mu_s\circ l(g)=l(\sigma_s(g))\circ\mu_s\quad\text{for all }s\in\Gamma,\ g\in G_\dmd(k).
\end{equation}
Similarly, $\tau_s(a)(\mu_s(y))=\mu_s(a(y))$, hence,
\begin{equation}\label{e:mu-a}
\mu_s\circ a=\tau_s(a)\circ\mu_s\quad\text{for all }s\in\Gamma,\ a\in \sA_\dmd\hs.
\end{equation}

By definition (Serre \cite[I.5.6]{Serre})
\[\delta[c]\in H^2(k_0,Z(G_\dmd))\]
is the class of the 2-cocycle given by
\[(\gamma,t)\mapsto \ctil_\gamma\cdot\upgam \ctil_t\cdot \ctil_{\gamma t}^{-1}\ \in Z(G_\dmd)(k) \quad (\gamma,t\in\Gamma).\]
Then $\vk_*(\delta[c])$ is the class of the 2-cocycle
\[(\gamma,t)\mapsto \vk( \ctil_\gamma\cdot\upgam \ctil_t\cdot \ctil_{\gamma t}^{-1})\in \sA_\dmd.\]

Let
\[ a\colon \Gamma\to\sA_\dmd\]
be a locally constant map.
We define
\[\mu^0_\gamma=a_\gamma\circ l(\ctil_\gamma)\circ\mu_\gamma=l(\ctil_\gamma)\circ a_\gamma\circ\mu_\gamma.\]

\begin{lemma}\label{l:equi}
For any $\gamma\in \Gamma$, the $\gamma$-semi-automorphism $\mu^0_\gamma$ is $\sigma^0_\gamma$-equivariant.
\end{lemma}

\begin{proof}
 Using \eqref{e:g-a} and  \eqref{e:mu-g}, we compute:
\begin{align*}
\mu^0_\gamma(g\cdot y)
&=(a_\gamma\circ l(\ctil_\gamma))(\mu_\gamma(g\cdot y))\\
&= a_\gamma( \ctil_\gamma\cdot\sigma_\gamma(g)\cdot\mu_\gamma(y)\hs )\\
&=\ctil_\gamma\hs \sigma_\gamma(g)\ctil_\gamma^{-1}\cdot a_\gamma(\hs \ctil_\gamma\cdot \mu_\gamma(y)\hs)
=\sigma^0_\gamma(g)\cdot\mu^0_\sigma(y).\qedhere
\end{align*}
\end{proof}

\begin{lemma}\label{l:hom}
The map $\gamma\to\mu^0_\gamma$ is a homomorphism
if and only if
\begin{equation}\label{e:neutral}
a_\gamma\cdot \upgam a_t\cdot\vk( \ctil_\gamma\upgam \ctil_t\,
            \ctil_{\gamma t}^{-1})\cdot a_{\gamma t}^{-1}=1  \quad\text{for all }\gamma,t\in\Gamma.
\end{equation}
\end{lemma}

\begin{proof}
Let $\gamma,t\in\Gamma$.
 Using \eqref{e:g-a}, \eqref{e:mu-g}, and \eqref{e:mu-a},  we compute:
\begin{align*}
\mu^0_\gamma\circ\mu^0_t \circ (\mu^0_{\gamma t})^{-1}
&=a_\gamma\circ l(\ctil_\gamma)\circ\mu_\gamma \circ a_t\circ l(\ctil_t)\circ\mu_t
\circ\mu_{\gamma t}^{-1}\circ l(\ctil_{\gamma t})^{-1}\circ a_{\gamma t}^{-1}\\
&=a_\gamma\circ \tau_\gamma(a_t)\circ l(\ctil_\gamma)\circ l(\sigma_\gamma(\ctil_t))
\circ\mu_\gamma\circ\mu_t\circ\mu_{\gamma t}^{-1}\circ l(\ctil_{\gamma t})^{-1}\circ a_{\gamma t}^{-1}.
\end{align*}
By \eqref{e:action-hom}  we obtain that
\begin{align*}
\mu^0_\gamma\circ\mu^0_t \circ (\mu^0_{\gamma t})^{-1}
&=a_\gamma\circ \tau_\gamma(a_t)\circ l(\ctil_\gamma)\circ l(\sigma_\gamma(\ctil_t))
\circ l(\ctil_{\gamma t})^{-1}\circ a_{\gamma t}^{-1}\\
&=a_\gamma\cdot \upgam a_t\cdot \vk(\ctil_\gamma \hs \upgam\ctil_t\,\ctil_{\gamma t}^{-1})\cdot a_{\gamma t}^{-1}.
\end{align*}
We see that
\[\mu^0_\gamma\circ\mu^0_t \circ (\mu^0_{\gamma t})^{-1}=1\]
if and only if \eqref{e:neutral} holds.
Thus the map $\gamma\to\mu^0_\gamma$ is a homomorphism if and only if \eqref{e:neutral} holds,
which completes the proof of Lemma \ref{l:hom}.
\end{proof}

Now assume that $\vk_*(\delta[c])\in H^2(\Gamma,\sA_\dmd)$ is neutral.
This means that there exists a locally constant map
\[a\colon \Gamma\to \sA_\dmd\]
such that \eqref{e:neutral} holds.
Then by Lemma \ref{l:hom}  the map
\[\mu^0\colon \Gamma\to\SAut(Y),\quad \gamma\mapsto\mu^0_\gamma\]
is a homomorphism, hence it satisfies hypothesis (i) of \cite{BG}, Lemma 6.3.
By Lemma \ref{l:equi} $\mu^0_s$ is $\sigma^0_s$-equivariant, hence
$\mu^0$ satisfies hypothesis (iv) of \cite{BG}, Lemma 6.3.
The variety $Y=G/H$ is quasi-projective, hence hypothesis (iii) of this lemma is satisfied.
It is easy to see that the restriction of the homomorphism $\mu_0$
to $\Gal(k/k_1)$ for some finite Galois extension $k_1/k_0$ in
$k$ comes from a $G_1$-equivariant $k_1$-model $Y_1$ of $Y$, where $G_1 = G_0 \times_{k_0} k_1$.
Thus the homomorphism $\mu^0$ satisfies hypothesis (ii) of  \cite{BG}, Lemma 6.3.
By this lemma the variety $Y$ admits a $G_0$-equivariant $k_0$-model $Y_0$ inducing the homomorphism $\gamma\mapsto\mu^0_\gamma$, as required.

Conversely, assume that $Y$ admits a $G_0$-equivariant $k_0$-model $Y_0$ inducing a homomorphism
\[\mu^0\colon\Gamma\to \SAut(G),\quad \gamma\mapsto\mu^0_\gamma.\]
Then by Lemma \ref{l:equi} (in the case $a_s=1$) the $\gamma$-semi-automorphism $l(c_\gamma)\circ\mu_\gamma$ of $Y$
is $\sigma^0_\gamma$-equivariant for any $\gamma\in\Gamma$.
Since $\mu^0_\gamma$ is $\sigma^0_\gamma$-equivariant as well, we have
\[\mu^0_\gamma=a_\gamma\circ l(c_\gamma)\circ\mu_\gamma\]
for some  locally constant map
\[a\colon \Gamma\to \sA_\dmd,\quad \gamma\mapsto a_\gamma.\]
Since the map $s\mapsto\mu^0_s$ is a homomorphism, by Lemma \ref{l:hom} the equality \eqref{e:neutral} holds
and hence, $\vk_*(\delta[c])$ is neutral in $H^2(\Gamma, \sA_\dmd)$.
This completes the proof of Theorem \ref{t:twist}.
\end{proof}

\section{Model of a homogeneous space of a reductive group}
\label{s:tits}

Let $k$, $k_0$, and $\Gamma$ be as in Subsection \ref{ss:intro-models}.
In this section $G$ is a connected reductive group over $k$.
Let $H\subset G$ be a $k$-subgroup (not necessarily spherical).
We consider the homogeneous $G$-variety $Y=G/H$.
Consider the abstract group $\sA=\Aut^G(G/H)$
and the algebraic group $A=\sN_G(H)/H$, then there is a canonical isomorphism $A(k)\isoto\sA$;
see e.g. \cite[Lemma 5.1]{BG}.
Let $G_\qs$ be a quasi-split $k_0$-model of $G$ and let $Y_\qs$ be a $G_\qs$-equivariant model of $G/H$,
then we obtain a $\Gamma$-action on $A(k)=\sA$ and hence, a $k_0$-model $A_\qs$ of $A$.
We need the following result:

\begin{proposition}\label{p:inn-qs}
Let $k_0\subset k$ be a subfield such that $k$ is a Galois extension of $k$.
Let $G$ be a connected reductive group over $k$.
Let $G_0$ be any $k_0$-model of $G$.
Then there exist a quasi-split inner $k_0$-model $G_\qs$ of $G$
 and a cocycle $c\in Z^1(k_0,\Inn(G_\qs)\hs)$ such that
$G_0\simeq \hs_c G_\qs$ (we say that $G_\qs$ is a quasi-split inner $k_0$-form of $G_0$).
Moreover, if $G_\qs$ and $G'_\qs$ are two quasi-split inner $k_0$-forms of $G_0$, then they are isomorphic.
\end{proposition}

\begin{proof}
See ``The Book of Involutions'' \cite[Proposition(31.5)]{KMRT}, or Conrad \cite[Proposition 7.2.12]{Conrad},
 the existence only in  Springer \cite[Proposition 16.4.9]{Springer2}.
\end{proof}

\begin{subsec}\label{ss:Tits}
Let $G$ be a connected reductive group over $k$, and let $G_0$ be a $k_0$-model of $G$.
Write $\Gbar=G/Z(G)$ for the corresponding adjoint group, and $\Gtil$ for the universal cover of
the connected semisimple group $[G,G]$.
By Proposition \ref{p:inn-qs} we may write  $G_0=\hs _c G_\qs$,
where $G_\qs$ is a quasi-split $k_0$-model of $G$ and $c\in Z^1(k_0, G_\qs/Z(G_\qs)\hs)$.
We fix $G_\qs$ and  $c$.
We write $\Gbar_\qs=G_\qs/Z(G_\qs)$.

We write $\Ztil_\qs$ for the center $Z(\Gtil_\qs)$ of the universal cover
$\Gtil_\qs$ of the connected semisimple group $[G_\qs,G_\qs]$.
Similarly, we write $\Ztil_0$ for the center $Z(\Gtil_0)$ of the universal cover
$\Gtil_0$ of the connected semisimple group $[G_0,G_0]$.
The short exact sequence
\[1\to \Ztil_\qs\to \Gtil_\qs \to \Gbar_\qs\to 1\]
induces a cohomology exact sequence
\[H^1(k_0,\Ztil_\qs)\to  H^1(k_0,\Gtil_\qs)\to H^1(k_0,\Gbar_\qs)\labelto{\deltatil} H^2(k_0,\Ztil_\qs).\]
By definition, the \emph{Tits class} $t(\Gtil_0)$  is the image of $[c]\in Z^1(k_0,\Gbar_\qs)$ in $H^2(k,\Ztil_\qs)$
under the connecting map $\deltatil\colon H^1(k_0,\Gbar_\qs)\to H^2(k_0,\Ztil_\qs)$;
compare \cite{KMRT}, Section 31, before Proposition (31.7).
\end{subsec}

\begin{theorem}\label{t:tits}
Let $G$ be a reductive group over an algebraically closed field $k$ of characteristic 0.
Let $H\subset G$ be an algebraic subgroup.
Let $k_0\subset k$ be a subfield such that $k$ is an algebraic closure of $k$.
Let $G_0$ be a $k_0$-model of $G$.
Write $G_0=\hs_c G_\qs$, where $G_\qs$ is a quasi-split inner form of $G_0$ and where $c\in Z^1(k_0,\Gbar_\qs)$.
Assume that $G/H$ admits a $G_\qs$-equivariant $k_0$-model.
Then $G/H$ admits a $G_0$-equivariant $k_0$-model if and only if
the image in $H^2(k_0,A_\qs)$ of the Tits class $t(\Gtil_0)\in H^2(k_0,Z(\Gtil_\qs))$ is neutral.
\end{theorem}

\begin{proof}
By Theorem \ref{t:twist} the homogeneous variety $G/H$ admits a $G_0$-equivariant $k_0$-form
if and only if the image
\[\vk(\delta[c])\in H^2(k_0,A_\qs)\]
is neutral.
We write  $Z_\qs$ for $Z(G_\qs)$ and  $\Ztil_\qs$ for $Z(\Gtil_\qs)$.
From the commutative diagram with exact rows
\[
\xymatrix{
1 \ar[r] & \Ztil_\qs \ar[r]\ar[d] & \Gtil_\qs \ar[r]\ar[d] & \Gbar_\qs \ar[r]\ar[d]^{\id} & 1 \\
1 \ar[r] & Z_\qs \ar[r]           &  G_\qs \ar[r]          & \Gbar_\qs \ar[r]             & 1
}
\]
we obtain a commutative diagram
\[
\xymatrix{
H^1(k_0,\Gbar_\qs)\ar[r]^{\deltatil} \ar[d]_{\id}  & H^2(k_0, \Ztil_\qs)\ar[d]^\lambda \\
H^1(k_0,\Gbar_\qs)\ar[r]^\delta                                   & H^2(k_0, Z_\qs),
}
\]
which shows that
\[ \delta[c]=\lambda(\deltatil[c]).\]
By definition
\[t(\Gtil_0)=\deltatil[c]\in  H^2(k_0,\Ztil_\qs).\]
Thus $\vk(\delta[c])$ is the image in $H^2(k_0,A_\qs)$ of $t(\Gtil_0)$
under the map
 \[ H^2(k_0,\Ztil_\qs)\to H^2(k_0,Z_\qs)\to  H^2(k_0,A_\qs)\]
  induced by the homomorphism
$\Ztil_\qs\to Z_\qs\to A_\qs$\hs.
We conclude that the homogeneous variety $G/H$ admits a $G_0$-equivariant $k_0$-form if and only if
the image  of  $t(\Gtil_0)$ in $H^2(k_0,A_\qs)$ is neutral, as required.
\end{proof}

\section{Models of  $(H\times H)/\Delta$} \label{s:H x H}

\begin{subsec}\label{ss:H-times-H/Delta}
Let $H$ be a connected algebraic $k$-group, and set $G=H\times_k H$.
Let $Y=H$, where $G$ acts on $Y$ by
\[ (h_1,h_2)*y=h_1\hs y \hs h_2^{-1}.\]
Note that $Y=G/\Delta$, where $\Delta\in H\times_k H$ is the diagonal,
that is, $\Delta$ is $H$ embedded in $G$ diagonally.
Let $H^{(1)}_0$ and $H^{(2)}_0$ be two $k_0$-models of $H$.
We set $G_0=H^{(1)}_0\times_{k_0} H^{(2)}_0$ and ask whether $Y$ admits a $G_0$-equivariant $k_0$-model.
\end{subsec}

\begin{theorem}\label{t:H-times-H}
With the notation and assumptions of \ref{ss:H-times-H/Delta},
$Y=(H\times H)/\Delta$ admits an $H^{(1)}_0\times_{k_0} H^{(2)}_0$-equivariant $k_0$-model
if and only if $H^{(2)}_0$ is a pure inner form of $H^{(1)}_0$.
\end{theorem}

\begin{proof}[Proof of Theorem \ref{t:H-times-H}]
Set $G_1=H^{(1)}_0\times_{k_0} H^{(1)}_0$, then $Y$ admits a $k_0$-model $Y_1= H^{(1)}_0$ (with the natural action of $G_1$).
Assume that $H^{(2)}$ is a pure inner form of $H^{(1)}$, then $G_0:=H^{(1)}_0\times_{k_0} H^{(2)}_0$ is a pure inner form of $G_1$,
and by Lemma \ref{l:pure-inner} the $G$-variety $Y$ admits a $G_0$-equivariant $k_0$-model.
Explicitly,  let $P_\ctil$ denote the \emph{torsor} (principal homogeneous space) of $H^{(1)}$
corresponding to the 1-cocycle $\ctil$; see Serre \cite[Section I.5.2]{Serre}.
Then $H^{(1)}$ acts on $P_\ctil$, and $H(k)=H^{(1)}(k)$ acts on $P$ simply transitively.
Moreover, $H^{(2)}=\hs_\ctil(H^{(1)})$ acts on $P_\ctil$ as well, and these two actions commute;
see Serre  \cite[Section I.5.3, Corollary of Proposition 34]{Serre}.
Thus $P_\ctil$ is an $H^{(1)}_0\times_{k_0} H^{(2)}_0$-equivariant $k_0$-model of $Y$.

Conversely, assume that $Y$ admits an $H^{(1)}_0\times_{k_0} H^{(2)}_0$-equivariant $k_0$-model.
First we show that then $H^{(2)}$ is an \emph{inner form} of $H^{(1)}$.
Indeed, let
\[\sigma^{(i)}\colon \Gamma\to\SAut(H) \]
denote the semilinear actions corresponding to the model $H_0^{(i)}$ of $H$ for $i=1,2$.
Recall that $\Delta(k)=\{(h,h)\ |\ h\in H(k)\}$.
Then for any $\gamma\in \Gamma$ we have
\[(\sigma^{(1)}_\gamma\times\sigma^{(2)}_\gamma)\hs(\Delta(k))=\{(\sigma^{(1)}_\gamma(h), \sigma^{(2)}_\gamma(h))\ |\ h\in H(k)\}.\]
Since $Y$ admits an $H^{(1)}_0\times_{k_0} H^{(2)}_0$-equivariant $k_0$-model,
 the subgroup $(\sigma^{(1)}_\gamma\times\sigma^{(2)}_\gamma)\hs(\Delta)$
is conjugate to $\Delta$ in $G=H\times H$; see, e.g., \cite[Lemma 4.1]{Stephan}.
This means that there exists a pair $(h_1,h_2)\in H(k)\times H(k)$ such that
\[(\sigma^{(1)}_\gamma(h), \sigma^{(2)}_\gamma(h))=(h_1 h h_1^{-1}, h_2 h h_2^{-1})\ \text{ for all }\ h\in H(k).\]
It follows that
\[\sigma^{(1)}_\gamma(h)=(h_1 h_2^{-1})\cdot \sigma^{(2)}_\gamma(h)\cdot(h_1 h_2^{-1})^{-1}.\]
We see that for any $\gamma\in\Gamma$, the $\gamma$-semi-automorphism  $\sigma^{(2)}_\gamma$ of $H$
differs from $\sigma^{(2)}_\gamma$ by an inner automorphism of $H$.
This means that $H^{(2)}_0$ is an inner form of $H^{(1)}_0$.

Now we know that $H^{(2)}_0=\hs_c (H^{(1)}_0)$ for some 1-cocycle $c\in Z^1(k_0,\Hbar^{(1)})$.
Set $G_1=H_0^{(1)}\times_{k_0}H_0^{(1)}$,
then  $Y_1:=H_0^{(1)}$ with the natural action of $G_1$ is a $G_1$-equivariant $k_0$-model of $Y$.
Moreover, $G_0=H^{(1)}_0\times_{k_0} H^{(2)}_0$ is the inner twisted form of $G_1$
given by the 1-cocycle $(1,c)\in Z^1(k_0, G_1)$.
Then
\[ \delta[(1,c)]\ \in\   H^2(k_0, Z(G_1))=H^2(k_0, Z( H_0^{(1)})\hs)\times H^2(k_0, Z(H_0^{(1)})\hs)\]
is $(1,\delta_H[c])$, where
\[\delta_H\colon H^1(k_0,\Hbar_0^{(1)})\to H^2(k_0, Z(H_0^{(1)})\hs)\]
is the connecting map.
By Theorem \ref{t:twist}, $Y$ admits an $H^{(1)}_0\times_{k_0} H^{(2)}_0$-equivariant $k_0$-model
if and only if $\vk_*(\delta[1,c])=0$,
that is,  if and only if $\vk_*(1,\delta_H[c])=1$.
An easy calculation shows that
\[\sN_G(\Delta)=Z(G)\cdot\Delta \quad \text{and}\quad A:=\sN_G(\Delta)/\Delta= Z(G)/Z(\Delta)=(Z(H)\times Z(H)\hs)\hs/\hs Z(\Delta).\]
Similarly, over $k_0$ we obtain
\[\sN_{G_1}(\Delta_1)=Z(G_1)\cdot\Delta_1 \ \text{and}\
A_1:=\sN_{G_1}(\Delta_1)/\Delta_1= Z(G_1)/Z(\Delta_1)=(Z(H^{(1)}_0)\times Z(H^{(1)}_0)\hs)\hs/\hs Z(H^{(1)}_0).\]
It is easy to see that the morphism of abelian $k_0$-groups
\[ Z(H_0^{(1)})\to A_1,\quad z\mapsto (1,z)\cdot Z(\Delta)\]
is an isomorphism.
It follows that the induced map on cohomology
\[ H^2(k_0,Z(H_0^{(1)})\hs)\to H^2(k_0, A_1)\]
is an isomorphism of abelian groups.
Therefore, $Y$ admits an $H^{(1)}_0\times_{k_0} H^{(2)}_0$-equivariant $k_0$-model
if and only if $\delta_H[c]=1$, that is,
if and only if $H^{(2)}_0$ is a pure inner form of $H^{(1)}_0$, as required.
\end{proof}

\begin{lemma}\label{l:Kneser}
Let $H_0$ be a simply connected semisimple group over a $p$-adic field $k_0$.
Then any pure inner form of $H_0$ is isomorphic to $H_0$.
\end{lemma}

\begin{proof}
Indeed, by Kneser's theorem   we have $H^1(k_0,H_0)=1$; see Platonov and Rapinchuk  \cite[Theorem 6.4]{PR}.
\end{proof}

\begin{corollary}
In Theorem \ref{t:H-times-H}, if $k_0$ is a $p${\emm-adic} field and $H$ is a
{\emm simply connected} semisimple group over $k$,
then $Y$ admits an $H^{(1)}_0\times_{k_0} H^{(2)}_0$-equivariant $k_0$-model
if and only if $H^{(2)}_0$ is isomorphic to $H^{(1)}_0$.
\end{corollary}

\begin{proof}
Indeed, by Theorem \ref{t:H-times-H} the variety $Y$
admits an $H^{(1)}_0\times_{k_0} H^{(2)}_0$-equivariant $k_0$-model if and only if
$H^{(2)}_0$ is a pure inner form of $H^{(1)}_0$, and
by Lemma \ref{l:Kneser} any pure inner form of $H^{(1)}_0$ is isomorphic to $H^{(1)}_0$.
\end{proof}

\section{Models of  $G/U$} \label{s:G/U}

\begin{theorem}\label{t:U}
Let $k$ be a fixed algebraic closure of a field $k$ of characteristic 0,
and let $G$ be a connected reductive group over $k$.
Let $B\subset G$ be a Borel subgroup, and write $U$ for the unipotent radical of $B$.
Consider the homogeneous space $Y=G/U$.
Let $G_0$ be a $k_0$-model of $G$.
Then $Y$ admits a $G_0$-equivariant $k_0$-model if and only if $G_0$ is a pure inner form of a quasi-split model of $G$.
\end{theorem}

\begin{proof}
It is well known that  $G_0$ is an inner form of a quasi-split model $G_\qs$ of $G$;
see  Springer \cite[Proposition 16.4.9]{Springer2}) or ``The Book of Involutions'' \cite[Proposition (31.5)]{KMRT},
or Conrad \cite[Proposition 7.2.12]{Conrad}.
This means that $G_0=\hs_c G_\qs$, where $c\in Z^1(k_0, \Gbar_\qs)$.
Since $G_\qs$ is quasi-split, there exists a Borel subgroup  $B_\qs\subset G_\qs$ (defined over $k_0$).
Set $U_\qs=R_u(B_\qs)$, then $G_\qs/U_\qs$ is a $G_\qs$-equivariant $k_0$-model of $Y=G/U$.
By Theorem \ref{t:twist}, $Y$ admits a $G_0$-equivariant $k_0$-model
if and only if $\vk_*(\delta[c])\subset H^2(k_0,A_\qs)$ vanishes,
where $A_\qs=\sN_G(U_\qs)/U_\qs\cong T_\qs$ and $T_\qs\subset B_\qs$ is a maximal torus.
Note that $\vk\colon Z_\qs\to A_\qs=T_\qs$ is the canonical embedding, where $Z_\qs=Z(G_\qs)$.

We show that the homomorphism
\[\vk_*\colon H^2(k_0,Z_\qs)\to H^2(k_0,T_\qs)\]
is injective.
Indeed, we have a short exact sequence
\[ 1\to Z_\qs\labelto{\vk} T_\qs\to\Tbar_\qs\to 1,\]
which induces a cohomology exact sequence
\[\dots\to H^1(k_0,\Tbar_\qs)\to H^2(k_0,Z_\qs)\labelto{\vk_*} H^2(k_0,T_\qs)\to\dots\]
Since $G_\qs$ is quasi-split, by Lemma \ref{l:H1-qs} below we have $H^1(k_0,\Tbar_\qs)=1$,
hence the homomorphism $\vk_*$ is injective, as required.

We see that $\vk_*(\delta[c])=1$ if and only if $\delta[c]=1$.
Now consider the cohomology exact sequence
\[\dots \to H^1(k_0,G_\qs)\to H^1(k_0,\Gbar_\qs)\labelto{\delta} H^2(k_0, Z_\qs)\to\dots\]
It follows from the construction of the map $\delta$ (see Serre \cite[Section I.5.6]{Serre})
that $\delta[c]=1$ if and only if $c$ can be lifted to a 1-cocycle $\ctil\in Z^1(k_0,G)$,
that is, if and only if $G_0=\hs_c G_\qs$ is a pure inner form of $G_\qs$, as required.

We conclude that $Y$ admits a $G_0$-equivariant $k_0$-model
if and only if $G_0$ is a pure inner form of $G_\qs$.
\end{proof}

\begin{lemma}\label{l:H1-qs}
Let $G_0$ be a quasi-split semisimple group of adjoint type,
 $B_0\subset G_0$ be a Borel subgroup defined over $k_0$, and $T_0\subset B_0$ be a maximal torus.
Then $H^1(k_0, B_0)=H^1(k_0, T_0)=1$.
\end{lemma}

\begin{proof}
Note that $T_0\simeq B_0/R_u(B_0)$,
which gives a canonical bijection $H^1(k_0,B_0)\isoto H^1(k_0, T_0)$;
see Sansuc  \cite[Lemme 1.13]{Sansuc}.
Since $G_0$ is a group of adjoint type, the set of simple roots
$S=S(G_{0,k},T_{0,k},B_{0,k})$ is a basis of the character group $\X^*(T_{0,k})$;
see Springer \cite[8.1.11]{Springer2}.
Since $B_{0,k}$ is defined over $k_0$, the action of $\Gamma$ on $\X^*(T_{0,k})$ preserves the basis $S$.
In other words,  $\X^*(T_{0,k})$ is a \emph{permutation} $\Gamma$-module,
hence $T_0$ is a \emph{quasi-trivial} $k_0$-torus, and therefore, $H^1(k_0,T_0)=1$;
see Sansuc \cite[Lemme 1.9]{Sansuc}.
\end{proof}

\begin{remark}
In Theorem \ref{t:U} assume that $G$ is a semisimple group of adjoint type.
Then a $k_0$-model of $G/U$, if exists, is unique.
Indeed, then $A_\qs\cong T_\qs$, and by Lemma \ref{l:H1-qs} we have
\[H^1(k_0, A_\qs)=H^1(k_0, T_\qs)=1.\]
\end{remark}

\begin{corollary}
In Theorem \ref{t:U} assume that $k_0$ is a $p$-adic field and that $G$
is semisimple and simply connected.
Then $G/U$ admits a $G_0$-equivariant $k_0$-model if and only if $G_0$ is quasi-split.
\end{corollary}

\begin{proof}
Indeed, by Theorem \ref{t:U} the $G$-variety $G/U$ admits a $G_0$-equivariant $k_0$-model
if and only if $G_0$ is a pure inner form of a quasi-split group $G_\qs$.
Since $k_0$ is a $p$-adic field, by Lemma \ref{l:Kneser} then $G_0$ is isomorphic to $G_\qs$.
\end{proof}

\end{document}